\documentclass[reqno]{amsart}
\usepackage{amsmath,amsfonts,amssymb,amsthm,epsfig,epstopdf,url,array,comment,hyperref,mathrsfs,mathtools,cancel}
\usepackage[normalem]{ulem}
\usepackage{hyperref}
\usepackage{setspace}
\usepackage{graphicx}
\usepackage{float}
\usepackage{tikz-cd}
\usepackage[cmtip,all]{xy}
\newcommand{\longsquiggly}{\xymatrix{{}\ar@{<~>}[r]&{}}}
\usepackage{tabularray}

\theoremstyle{plain}
\newtheorem{Theorem}{Theorem}[section]

\newtheorem{Lemma}[Theorem]{Lemma}
\newtheorem{Proposition}[Theorem]{Proposition}

\theoremstyle{Definition}
\newtheorem{Definition}[Theorem]{Definition}
\newtheorem{Remark}[Theorem]{Remark}

\newtheorem{Question}[Theorem]{Question}

\newcommand{\ind}{\mathbf 1}

\newcommand{\R}{\mathbb{R}}
\renewcommand{\d}{\textnormal{ d}}

\newcommand{\diam}{\textnormal{diam}}

\newcommand{\hidethis}[1]{}

\newcommand{\support}{\textnormal{supp}}

\newcommand{\supp}{\textnormal{support}}

\begin{document}

\title{Metric complexity is a Bryant--Tupper diversity}
\author{Gautam Aishwarya}
\address{Michigan State University, Department of Mathematics, East Lansing, MI 48824, USA.}
\email{aishwary@msu.edu}
\author{Dongbin Li}
\address{University of Alberta, Department of Mathematical and Statistical Sciences, Edmonton, AB T6G 2N8, Canada.}
\email{dongbin@ualberta.ca}
\author{Mokshay Madiman}
\address{University of Delaware, Department of Mathematical Sciences, Newark, DE
19716, USA.}
\email{madiman@udel.edu}
\author{Mark Meckes}
\address{Case Western Reserve University, Department of Mathematics, Applied Mathematics, and Statistics, Cleveland, OH 44106, USA.}
\email{mark.meckes@case.edu}

\begin{abstract}
The metric complexity (sometimes called Leinster--Cobbold maximum diversity) of a compact metric space is a recently introduced isometry-invariant of compact metric spaces which generalizes the notion of cardinality, and can be thought of as a metric-sensitive analogue of maximum entropy. On the other hand, the notion of diversity introduced by Bryant and Tupper is an assignment of a real number to every finite subset of a fixed set, which generalizes the notion of a metric. We establish a connection between these concepts by showing that the former quantity naturally produces an example of the latter. Moreover, in contrast to several examples in the literature, the diversity that arises from metric complexity is Minkowski-superadditive for compact subsets of the real line.
\end{abstract}
\thanks{ 
	\\\indent GA is supported by NSF-DMS 2154402.
}

\maketitle
\section{Introduction and main results}
In \cite{BT12:1}, Bryant and Tupper considered the following class of set functions, called \emph{diversities}, as a generalization of the notion of a metric.
\begin{Definition} \label{def: BTdiv}
Let $X$ be a set, and $\mathcal{F}(X)$ be the collection of all finite subsets of $X$. A set function $\delta:\mathcal{F}(X)\rightarrow \R$  is called a diversity, if it satisfies the properties:
\begin{enumerate}
    \item $\delta (A) = 0 $ if  and only if the cardinality of $A$ is at most $1$ \emph{(non-degeneracy)}, and
    \item $\delta (A \cup B) \leq \delta (A \cup C) + \delta (B \cup C)$, whenever $C \neq \emptyset$ \emph{(triangle inequality)}. 
\end{enumerate}
\end{Definition}
Indeed, the definition of a metric on a set $X$ can be thought of as a set function satisfying the same properties as above, but defined only on the collection of all subsets of cardinality at most $2$. Moreover, the diameter, 
\[
\diam (A) = \max_{x,y \in A} d (x, y),
\]
is an example of a diversity that extends a given metric $d$ on $X$. 

In \cite{BT12:1}, diversities were introduced to extend the notion of hyperconvexity from metrics to this richer class of set functions. Since then, a number of works have appeared in the direction of extending various aspects of metric space theory to diversities (for example, \cite{BNT17,HN20,HGN22,EP14}). We learn from \cite{BT12:1} that the name diversity was chosen because a special case of the definition appears in the literature on phylogenetics and ecological diversity. 

Approximately around the same time as \cite{BT12:1} appeared, motivated from an ecological diversity perspective, Leinster and Cobbold \cite{LC12} introduced a generalization of R\'enyi entropies of order $\alpha$. Fix a set $X$ and a function $Z : X \times X \to [0,1]$, which will also be considered as a matrix in the sequel. $Z$ is called a \emph{similarity kernel}, and $Z(x,y)$ measures, on scale from $0$ to $1$, how ``similar'' the point $x$ is to the point $y$. $Z(x,y)=0$ is to be understood as complete dissimilarity and $Z(x,y)=1$ indistinguishability.
\begin{Definition} \label{def: kernelalphacomplex}
    Given a probability measure $p = (p_{x})_{x \in X}$ on a finite set $X$, the \emph{kernelized $\alpha$-complexity of $p$ with respect to the kernel $Z$} is defined by 
    \begin{equation}\label{eq:q-comp}
H^{Z}_{\alpha}(p) =
\begin{cases}
\frac{1}{1 - \alpha} \log \left( \sum_{x} p_{x}(Zp)_{x} ^{\alpha - 1}  \right) , & \textnormal{ if } \alpha \in [0,1) \cup (1, \infty), \\
- \sum_{x} p_{x} \log (Zp)_{x}, & \textnormal{ if } \alpha = 1, \\
- \log \left( \max_{x \in \supp(p)} (Zp)_{x} \right), & \textnormal{ if } \alpha = \infty, \\
\end{cases}
    \end{equation}
    where $(Zp)_{x} = \sum_{y \in X} Z(x,y)p_{y}$. 
\end{Definition}

    The quantity $e^{H_{\alpha}^{Z}(p)}$ has been called \emph{the diversity of order $\alpha$ of $p$ with respect to $Z$} in the literature. Here we suggest the nomenclature of ``complexity'' for $H_{\alpha}^{Z}(p)$ to avoid confusion with ``diversity'' in the sense of Bryant and Tupper. 
\begin{Remark}
The Kronecker delta kernel $Z(x,y) = \delta_{x,y}$ gives R\'enyi entropies:
\[
H^{Z}_{\alpha}(p) = H_{\alpha}(p) : = \frac{1}{1 - \alpha} \log \left ( \sum_{x} p_{x}^{\alpha} \right).
\]
This corresponds to the situation when any two points $x$ and $y$ are either the same or completely dissimilar. 
\end{Remark}

Cardinality is related to maximum entropy in the following manner. Let $X$ be a finite set. Then $\max_{p \in \mathcal{P}(X)} H_{\alpha}(p) = \log \#(X)$, where $\mathcal{P}(X)$ is the collection of all probability measures on $X$ and $\# (\cdot)$ denotes cardinality. Note that this representation of cardinality holds for any fixed choice of $\alpha$.

In a similar spirit, a cardinality-like notion that is sensitive to similarity can be introduced.

\begin{Definition} \label{def: maxdiv}
    Let $X$ be a finite set and $Z$ a similarity kernel on $X$ which is symmetric (that is, $Z(x,y) = Z(y,x)$ for all $x,y \in X$), and satisfies $Z(x,x) > 0$ for all $x$. Then, the \emph{kernelized complexity of $X$ with respect to $Z$} is defined by 
\begin{equation}
        C^{Z}(X) = \sup_{p \in \mathcal{P}(X)} H_{\alpha}^{Z}(p).
\end{equation}
\end{Definition}
Thanks to a result of Leinster and Meckes \cite[Theorem 1]{LM16}, the choice of $\alpha$ in the above definition does not matter. 

From this point onwards, we will fix a metric space $(X, d)$ (not necessarily consisting of finitely many points), and work with the similarity kernel $Z (x,y) = e^{-  d (x,y)}$,
which is called the \emph{Laplace kernel}.

\begin{Definition}
    Let $(X,d)$ be a metric space, and $A$ be a finite subset of $X$. The complexity of $A$, denoted $C(A)$, is defined as the kernelized complexity of $A$ with respect to the Laplace kernel $Z(x,y) = e^{- d (x,y)}$.
\end{Definition}

When our set is already equipped with a metric, a key reason for restricting attention to the Laplace kernel is that the kernelized complexity for this kernel is closely related to another invariant called \emph{magnitude}. If one thinks of a metric space as an enriched category in the sense of Lawvere \cite{Law73}, a natural Euler characteristic-type invariant associated with enriched categories specializes to magnitude. This is explained in the foundational work \cite{Lei13} of Leinster. The notion of complexity is obtained from log-magnitude by altering its definition slightly, and is arguably a more tractable invariant with a probabilistic interpretation. We recommend the works \cite{LM17:1, Mec15} for an overview of magnitude, including its relationship with complexity.

\begin{Definition} \label{def: complexitymetricspace}
Fix $t > 0$. For every finite subset $A \subseteq X$, 
define $C^{t}(A)$  as the complexity of the metric space $(A , t d)$. For all compact sets $K \subseteq X$, define 
\[
C^{t}(K) = \sup \{ C^{t} (A) : A \subseteq K, \#(A) < \infty \}.
\]
By convention, we set $C^t(\emptyset)=0$ for the empty set.
\end{Definition}

    As a consequence of this definition, the complexity $C(X)=C^{1}(X)$ is well defined for any compact metric space $X$. The complexity of a compact metric space $X$ can also be defined directly \`a la Definition \ref{def: maxdiv}, using the notion of kernelized $\alpha$-complexity for probability measures that are not necessarily finitely supported, which is obtained by replacing the sums in Definition \ref{def: kernelalphacomplex} with integrals. For more details, see \cite[Definition 4.1]{LR21}-- a key point in the general setting is that the object $Zp$ obtained by applying the Laplace kernel to any probability measure is always a continuous function that is essentially bounded with respect to $p$, so that the generalization of \eqref{eq:q-comp} to arbitrary probability measures on a compact metric space makes sense. Sometimes we will write $Zp (x)$ instead of $\left(Zp\right)_{x}$, especially when $p$ is not finitely supported. The equivalence of the two definitions of complexity discussed here was first established in \cite[Theorem 2.4]{Mec13} (from the proof one can see that the assumption of \emph{positive-definiteness} stated there is not needed).

Our main result establishes a connection between the notions of metric complexity and the diversity of Bryant--Tupper.

\begin{Theorem} \label{thm: main}
    Let $(X, d)$ be a metric space, and $t > 0$. Then
    both 
\begin{equation} \label{eq: divfroment}
\kappa^t:=\exp \{C^{t}\}- 1
\end{equation}
    and $C^t$ 
    are diversities (in the sense of Definition \ref{def: BTdiv}).
\end{Theorem}
\begin{Remark}
The diversities $\kappa^{t}(\cdot)$ are intrinsic, that is, $\kappa^{t} (A)$ does not depend on the isometric embedding $A \hookrightarrow X$ but only on the metric structure of $A$. Thus, $\kappa^{t}$ can be seen as a numerical invariant on the class of finite metric spaces. 
\end{Remark}
\begin{Remark}
   The restriction of the diversity $\kappa^{t}$ to two element subsets defines a metric, $\tilde{d}^{t}$. Explicitly, we have, $\tilde{d}^{t} = \tanh \frac{td}{2}$. This follows from \cite[Examples 2.1.1 (ii)]{Lei13}. Note that \cite{Lei13} is about the invariant \emph{magnitude}, however, magnitude and metric complexity agree for compact metric spaces isometric to a subset of the real line.
\end{Remark}
Having proved that metric complexity $C^t$ naturally induces the diversity $\kappa^t$, we now proceed to establish two interesting properties of these quantities. The first relates to the behavior of $\kappa^t$ under Minkowski summation when the ambient space is the real line, while the second relates to its behavior under taking unions in a general metric space.

When the underlying metric space $(X,d)$ is assumed to be $\ell_{2}^{n}$, that is, $\R^n$ with the standard Euclidean metric, many examples of diversities satisfy additivity or subadditivity with respect to Minkowski sums $A + B = \{ a + b : a \in A, b \in B \}$; we henceforth call these properties Minkowski-additivity or Minkowski-subadditivity, respectively. The recent work \cite{BT24} is devoted to the study of diversities that are \emph{Minkowski-sublinear}, that is, Minkowski-subadditive and $1$-homogeneous. The examples constructed in the present work demonstrate different behaviour, namely Minkowski-superadditivity and inhomogeneity, as reflected by the following result in dimension one. 

\begin{Theorem} \label{thm: cdbm}
    Let $A,B \subseteq \R$ be non-empty and compact. Then 
    \begin{equation} \label{eq: CauchyDavenport}
\kappa^{t} (A + B) \geq \kappa^{t} (A)  + \kappa^{t} (B), 
    \end{equation}
    for all $t$. Further, if $\lambda \in [0,1]$, then we have
\begin{equation} \label{eq: BrunnMinkowski}
        \kappa^{t} ((1- \lambda) A + \lambda B ) \geq (1-\lambda) \kappa^{t} (A) + \lambda \kappa^{t} (B). 
    \end{equation}
\end{Theorem}

Finally, we demonstrate a fractional subadditivity property of complexity with respect to unions.

\begin{Definition}
Denote by $2^{[n]}$ the collection of all subsets of $[n]= \{ 1, \ldots , n \}$. A function $\beta: 2^{[n]} \rightarrow \R_+ $ is called a fractional partition if for each $i\in [n]$, we have $\sum_{s\in 2^{[n]}} \beta(s)\ind_{s} (i) =1.$
A set function $f: 2^{[n]} \rightarrow \R$ is said to be fractionally subadditive if 
$$
f([n])\leq \sum_{s \in 2^{[n]}} \beta (s) f(s) 
$$
for all fractional partitions $\beta$.
\end{Definition}

Fractional subadditivity is a stronger property than (usual) subadditivity; indeed, it reduces to subadditivity when $\beta$ is chosen such that $\beta(\{ i \}) = 1$ for all singletons $\{i \}$ and $0$ otherwise. Fractional subadditivity arises naturally in connection with entropy; see, e.g., \cite{MT10, MG19}.

\begin{Theorem} \label{thm: subadditivecomplex}
Let $(X,d)$ be a metric space, $A_1, A_2, \cdots, A_n  \subseteq X$ be compact subsets, and $\beta: 2^{[n]} \to [0, \infty)$ be a fractional partition. Then, for every $t >0$,
\[ \exp C^{t} \left( \bigcup_{i=1}^n A_i \right)  \leq \sum_{s \in 2^{[n]}} \beta (s) \exp C^{t} \left( \bigcup_{i \in s} A_i \right).\]
In other words, given any compact subsets $A_1, A_2, \cdots, A_n$ of $X$,
the set function $f(s):=1+\kappa^t\left( \bigcup_{i\in s} A_i \right)$ is fractionally subadditive.
\end{Theorem}
\begin{Remark}
    We may express the conclusion of Theorem~\ref{thm: subadditivecomplex} as the statement that $1+\kappa^t$
    is fractionally subadditive with respect to unions.
    Note that in contrast, for a general Bryant--Tupper diversity $\delta$, the set function $1+\delta$ need not even be subadditive with respect to unions. For example, consider the diameter diversity $\delta = \diam$. To see why $1+ \delta$ is not subadditive, take non-empty sets $A, B$, a unit vector $v$, and look at $\diam \left( (A + tv) \cup (B - tv) \right)$ as $t \to \infty$. 
\end{Remark}
The rest of the paper is organized as follows. In Section~\ref{sec:wedge}, we give the proof of our main result, namely Theorem~\ref{thm: main}, as a consequence of a new observation about the wedge sum of pointed metric spaces. In Section~\ref{sec:prop}, we prove the properties of complexity described in Theorems \ref{thm: cdbm} and \ref{thm: subadditivecomplex}. The proof of the latter goes through a corresponding subadditivity property for exponentiated $\alpha$-complexity of probability measures (Proposition~\ref{prop: fracsubmultcomplex}), which may be of independent interest. Finally we conclude in Section~\ref{sec:concl} with some remarks and discussion of open questions.

\subsection*{Acknowledgments} We would like to thank the anonymous referee for their valuable comments and suggestions, in particular, for suggesting a simpler proof of the reverse implication in Lemma \ref{lem: requirementtheirdivourdiv}.

\section{The wedge sum and the proof of Theorem~\ref{thm: main}}
\label{sec:wedge}

We start with two general observations about diversities that may of independent interest. The first provides a recipe, given a particular diversity, for generating a new diversity that ``grows more slowly'' when moving to supersets.

\begin{Lemma}\label{lem:derived-div}
If $\delta (\cdot)$ is a diversity, then so is $\widetilde{\delta} = \log \left( \delta + 1 \right)$.    
\end{Lemma}
\begin{proof}
The set function $\widetilde{\delta}$ clearly satisfies the first requirement of Definition \ref{def: BTdiv}. For the second condition, we note that $\delta (A \cup C) \delta (B \cup C) \geq 0$, and hence, $\delta (A \cup B) \leq \delta (A \cup C) + \delta (B \cup C) + \delta (A \cup C) \delta (B \cup C)$ whenever $C$ is not empty. Now, the last inequality can be easily seen to be equivalent to the second requirement in Definition \ref{def: BTdiv} for $\widetilde{\delta}$. 
\end{proof}

Bryant and Tupper \cite{BT24} observe that if a set function satisfies the non-degeneracy condition in  Definition \ref{def: BTdiv}, the triangle inequality condition is equivalent to the combination of monotonicity and subadditivity for sets with nonempty intersection. 
Our second general observation refines their observation, providing a simpler way of verifying that a given set function is a diversity. 

\begin{Lemma} \label{lem: requirementtheirdivourdiv}
Suppose the set function $\delta:\mathcal{F}(X)\rightarrow \R$ satisfies the non-degeneracy condition (i.e., $\delta(A)=0$ iff the cardinality of $A$ is 0 or 1) and is monotone (i.e., $\delta(A)\leq \delta(B)$ if $A\subset B$).
Then $\delta$ is a diversity, if and only if,
\begin{equation}\label{eq:subadd}
\delta( A \cup B )  \leq \delta (A) + \delta (B)
\end{equation}
for all pairs $A, B \in \mathcal{F}(X)$ that intersect in exactly one point. 
\end{Lemma}
\begin{proof}
For the forward implication, simply apply the triangle inequality for $\delta$ with the choice of $C = A \cap B$, for sets $A$ and $B$ that intersect in exactly one point.

For the backwards implication, assume that $\delta( A \cup B )  \leq \delta (A) + \delta (B)$
for non-empty finite subsets $A, B \subseteq X$ that intersect in exactly one point. For $\delta$ to be a diversity, the inequality $\delta (A \cup B) \leq \delta (A \cup C) + \delta (B \cup C)$, for finite sets $A,B,C$ such that $C \neq \emptyset$, must be verified. For $c \in C$, let $\tilde{A} = (A \setminus B) \cup \{c\}$, $\tilde{B} = B \cup \{c\}$. Then, 
\[\delta(A \cup B) \leq \delta(\tilde{A} \cup \tilde{B}) \leq \delta(\tilde{A}) + \delta(\tilde{B}) \leq \delta(A \cup C) + \delta(B \cup C),\]
thus completing the proof.

\end{proof}

We now proceed toward the proof of Theorem \ref{thm: main}. We wish to prove that for any given metric space $(X,d)$, the set function $\kappa^t:=\exp \{C^{t}\}- 1$ defined on $\mathcal{F}(X)$ is a diversity. 
Let us observe at the outset that by Lemma~\ref{lem:derived-div}, proving this would immediately imply the fact that the metric complexity $C^t$ is also a diversity. 

We fix a metric space $(X,d)$. For notational convenience, we write $D^{t}(\cdot ) = \exp C^{t} (\cdot)=\kappa^t +1$. Without loss of generality, we will assume $t = 1$ and drop the superscript $t$ in this section. 

It is clear from the defining equation \eqref{eq: divfroment}
that $\kappa$ is both non-degenerate and monotone, inheriting these properties from the metric complexity $C$. Therefore, by Lemma~\ref{lem: requirementtheirdivourdiv}, it suffices to prove that 
\begin{equation*}
\kappa( A \cup B )  \leq \kappa (A) + \kappa (B)    
\end{equation*}
for all pairs $A, B \in \mathcal{F}(X)$ that intersect in exactly one point. Since $\kappa( \cdot ) = D ( \cdot ) - 1$,
the inequality to be verified reads
\begin{equation}\label{eq:subadd-for-D}
D( A \cup B ) + 1 \leq D (A) + D (B )    
\end{equation}
for non-empty finite subsets $A, B \subseteq X$ that intersect in exactly one point. 

Our proof of  Theorem~\ref{thm: main} rests on the notion of the wedge sum of pointed metric spaces.
Recall that a pointed metric space $(X, d, x_{0})$ is simply a metric space $(X,d)$ with a distinguished point $x_{0}$. The wedge sum of two pointed metric spaces is defined by gluing them at their distinguished points. 
\begin{Definition}
For two pointed metric spaces, $X_{1} = (X_{1}, d_{1} , x_{1} ), X_{2} = (X_{2}, d_{2} , x_{2} )$ we define their wedge sum $X_{1} \vee X_{2}$ as pointed metric space $(X , d , x_{0})$ on the underlying set $ X = X_{1} \sqcup X_{2} / x_{1} \sim x_{2}$, equipped with the metric
\[
d(x,y)=
\begin{cases}
d_{i} (x,y) & \textnormal{ if } x,y \in X_{i}, \\
d_{i} (x,x_{i}) + d_{j} (x_{j}, y) & \textnormal{ if } x \in X_{i}, y \in X_{j}, i \neq j,\\
\end{cases}
\]
with the distinguished point $x_{0}$ which is the equivalence class of the $x_{i}$. 
\end{Definition}

The reason for the usefulness of the wedge sum for our purposes is that it is the hardest structure to verify the condition \eqref{eq:subadd-for-D} for $\kappa$ to be a diversity, as demonstrated by the following lemma. 

\begin{Lemma}\label{lem:suff-wedge}
Suppose that for all non-empty finite metric subspaces $A,B \subseteq X$ with arbitrary choices of distinguished points, the inequality
\begin{equation}
D(A \vee B ) + 1 \leq D( A ) + D(B) ,
\end{equation}
or equivalently $\kappa( A \vee B )  \leq \kappa (A) + \kappa (B)$,
is satisfied. Then $\kappa$ is a diversity.
\end{Lemma}
\begin{proof}
By Lemma \ref{lem: requirementtheirdivourdiv}, we only need to consider finite subsets $A,B \subseteq X$ which intersect in exactly one point, say $x_{0}$. Consider $A,B$ as pointed metric spaces with $x_{0}$ as the distinguished point. The set-theoretic identity map $A \vee B \to A \cup B$ is then $1$-Lipschitz. It is easy to see that the metric complexity is non-increasing under $1$-Lipschitz maps. Hence, 
\[
D ( A ) + D(B ) - 1 \geq D( A \vee B ) \geq D(A \cup B ). \qedhere
\]
\end{proof}

In view of Lemma~\ref{lem:suff-wedge}, Theorem \ref{thm: main} would be an immediate consequence of the following interesting property of $\kappa^t$ for wedge sums.

\begin{Theorem} \label{thm: wedgediv}
    Let $A$ and $B$ be non-empty finite pointed metric spaces. Then, 
    \begin{equation} \label{eq: submodwhenCis1}
        \kappa^{t} (A \vee B) \leq \kappa^{t} (A) + \kappa^{t} (B).
    \end{equation}
\end{Theorem}

\begin{proof}
We shall write the proof in the language of metric complexity. For convenience, we write $D_{\alpha}(\cdot)$ for $\exp H^{Z}_{\alpha} (\cdot)$.

Let $\rho$ be a metric complexity-attaining probability measure on $A \vee
B$, and denote by $x_{0}$ the common point of $A$ and $B$ in $A \vee B$.
We will consider two cases, according to whether $x_{0} \in
\operatorname{supp} \rho$.

Suppose first that $x_{0} \in \operatorname{supp} \rho$. Then we can write
$\rho = \mu + \nu + \epsilon \delta_{x_{0}}$, where $\mu$ is supported on
$A \setminus \{x_{0}\}$, $\nu$ is supported on $B \setminus \{x_{0}\}$, and
$\epsilon > 0$. We denote $a = \sum_{x} \mu (x) , b = \sum_{x} \nu (x)$,
\[
  M = Z\mu(x_{0}) = \sum_{x} e^{-d(x,x_{0})} \mu(x), \quad
  N = Z\nu(x_{0}) = \sum_{x} e^{-d(x,x_{0})} \nu(x),
\]
and let $R = D( A \vee B )$.  We have $Z\rho(x) = 1/R$ for each $x
\in \operatorname{supp} \rho$ \cite[Lemma 3]{LM16},
which implies that
\[
  \frac{1}{R} =
  \begin{cases}
    Z\mu(x) + e^{-d(x,x_{0})} (N + \epsilon) = Z(\mu + (N + \epsilon) \delta_{x_{0}})(x)
    & \text{ if } x \in
    \operatorname{supp} \mu, \\
    Z\nu(x) + e^{-d(x,x_{0})} (M + \epsilon) = Z(\nu + (M + \epsilon) \delta_{x_{0}})(x)
    & \text{ if } x \in
    \operatorname{supp} \nu, \\
    M + N + \epsilon & \text{ if } x = x_{0}.
  \end{cases}
\]

Now define $\widetilde{\mu} = \mu + (N+\epsilon) \delta_{x_{0}}$ and
$\widetilde{\nu} = \nu + (M + \epsilon) \delta_{x_{0}}$.  Then $Z \widetilde{\mu} (x)
= M + N + \epsilon = \frac{1}{R}$ for all $x \in \operatorname{supp}
\widetilde{\mu} = (\operatorname{supp} \mu) \cup \{x_{0}\}$, and also
$Z \widetilde{\nu} (x)
= N + M + \epsilon = \frac{1}{R}$ for all $x \in \operatorname{supp}
\widetilde{\nu} = (\operatorname{supp} \nu) \cup \{x_{0}\}$. Therefore,
\[
\begin{split}
  D (A ) + D( B ) &\ge D_{0}( \tilde{\mu}/(a + N + \epsilon)) + D_{0}(\tilde{\nu}/(b + M + \epsilon)) = R \sum_{x \in \support \widetilde{\mu}} \widetilde{\mu}(x)  + R
   \sum_{x \in \support \widetilde{\nu}} \widetilde{\nu} (x) \\ &= R (a + N + \epsilon + b + M + \epsilon ) = R + 1, \\
   \end{split}
\]
since $a + b + \epsilon = 1$ and $M + N + \epsilon = 1/R$.

\medskip

Now suppose that $x_{0} \notin \operatorname{supp} \rho$.  In this case we
can similarly write $\rho = \mu + \nu$, and we have
\[
  \frac{1}{R} =
  \begin{cases}
    Z\mu(x) + e^{-d(x,x_{0})} N  = Z(\mu + N \delta_{x_{0}})(x)
    & \text{ if } x \in
    \operatorname{supp} \mu, \\
    Z\nu(x) + e^{-d(x,x_{0})} M = Z(\nu + M \delta_{x_{0}})(x)
    & \text{ if } x \in
    \operatorname{supp} \nu.
  \end{cases}
\]

We let $\widetilde{\mu} = \mu + N\delta_{x_{0}}$ and
$\widetilde{\nu} = \nu + M \delta_{x_{0}}$, then estimate
\[
 D_{0} (\widetilde{\mu}/(a + N)) = \sum_{x \in \support \widetilde{\mu}}
  (Z\widetilde{\mu})^{-1} (x) \widetilde{\mu} (x)
  = Ra + N (Z \widetilde{\mu})(x_{0})^{-1} = Ra + \frac{N}{M+N}, 
\]
and similarly
\[
  D_0(\widetilde{\nu}/(b+M))=\sum_{x \in \support \widetilde{\nu}}
  (Z\widetilde{\nu})^{-1} (x) \widetilde{\nu} (x) = Rb + \frac{M}{M+N}.
\]
Therefore
\[
  D(A) + D(B ) \ge D_0(\widetilde{\mu}/(a + N)) +
  D_0(\widetilde{\nu}/(b+M)) = R + 1
\]
since $a + b = 1$.   
\end{proof}
Thus, the proof of Theorem \ref{thm: main} is concluded.

\section{Two properties of metric complexity}
\label{sec:prop}

We first prove the Minkowski-superadditivity of the diversity $\kappa$ arising from complexity, which is asserted by Theorem \ref{thm: cdbm}.

\begin{proof}[Proof of Theorem \ref{thm: cdbm}]
Note that $\kappa (\cdot)$ is translation invariant. Thus, without loss of generality, the sets $A$ and $B$ can be assumed to satisfy $\max_{x \in A} x = 0 = \min_{y \in B} y$. Under this assumption, we have $A \cup B \subseteq A + B$, and $A \cup B$ is isometric to the wedge sum $(A, 0) \vee (B,0)$. Thus,
\[
\kappa (A + B) \geq \kappa (A \cup B) = \kappa (A \vee B) = \kappa (A) + \kappa (B),
\]
where the last equality is due to \cite[Corollary 2.3.3]{Lei13}.
This proves the first part. For the second part, we use an explicit formula for $D(E)$ \cite[Theorem 4.1]{LM17:1}, for compact $E \subseteq \R$, to observe that the function $\lambda \mapsto  D (\lambda \cdot E ) -
(\lambda D(E) + (1 - \lambda ))$ is always concave. This leads to inequalities
\[
\quad  D((1- \lambda) \cdot A) \geq (1-\lambda)D(A) + \lambda, D(\lambda \cdot B) \geq \lambda D(B) + (1 - \lambda),
\]
when applied to $E =A, E=B$, respectively, since a concave function of $\lambda$ equal to zero at the endpoints $\lambda = 0, \lambda =1$ must stay non-negative in the interval $[0,1]$. Applying the conclusion of the first part to the sets $ (1-\lambda ) A $ and $\lambda B$,
\[
\begin{split}
D( (1- \lambda) A + \lambda B ) 
& \geq D( (1- \lambda) A) + D( \lambda B ) - 1 \\
& \geq (1-\lambda)D(A) + \lambda +  \lambda D(B) + (1 - \lambda) - 1 \\
& = (1-\lambda)D(A) + \lambda D(B).
\end{split}
\]
When written in terms of $\kappa (\cdot)$, this is exactly the inequality \eqref{eq: BrunnMinkowski}.
\end{proof}

Our investigation of fractional subadditivity proceeds via the corresponding property for exponentiated $\alpha$-complexity of probability measures. 

\begin{Proposition}\label{prop: fracsubmultcomplex}
Let $(X,d)$ be a metric space equipped with a similarity kernel $Z(x,y) = e^{- d (x,y)}$. Let $\mu_1, \ldots, \mu_n$  be probability measures on $X$, and  $\mu$ be a mixture of them: 
$$\mu =\sum_{i=1}^{n} \lambda_i \mu_i,$$ 
where $\lambda_i \geq 0$ for each $i$ and $\sum_{i=1}^{n} \lambda_i=1$.
Then, for any $ \alpha \in [0,\infty]$, 
and any fractional partition $\beta$ on $[n]$,
\[ \exp H^{Z}_{\alpha}(\mu) \leq \sum_{s \in 2^{[n]}} \beta (s) \exp H^{Z}_{\alpha}(\mu_s),\]
where $\mu_s=\frac{\sum_{i \in s} \lambda_i \mu_i }{\sum_{i \in s} \lambda_i}.$
\end{Proposition}

\begin{proof}
For a probability measure $\mu$, as mentioned earlier, we write $Z \mu (x) = \int Z (x,y) \d \mu (y)$. For $ \alpha \in [0,1) \cup (1, \infty)$, we have
\begin{align*}
D_{\alpha}(\mu)
& = \left( \int (Z \mu)^{\alpha-1} \d \mu \right) ^{1/(1-\alpha)} 
=\left(\sum_{i \in [n]} \lambda_i \int (Z \mu)^{\alpha-1} \d \mu_i \right) ^{1/(1-\alpha)} \\
&= \left(\sum_{i \in [n]} \left( \sum_{s\in 2^{[n]}: i \in s} \beta(s) \right) \lambda_i \int (Z \mu)^{\alpha-1} \d \mu_i \right) ^{1/(1-\alpha)}\\
& = \left( \sum_{s\in 2^{[n]}} \beta(s) \sum_{i \in s} \bigg\{ \lambda_i \int (Z \mu)^{\alpha-1} \d \mu_i \bigg\} \right) ^{1/(1-\alpha)} \\
& \leq  \left( \sum_{s\in 2^{[n]}} \beta(s) \sum_{i \in s} \bigg\{ \lambda_i \int \bigg( \sum_{j \in s} \lambda_j Z \mu_j \bigg)^{\alpha-1} \d \mu_i \bigg\} \right) ^{1/(1-\alpha)} ,
\end{align*}
where to justify the inequality we use $\sum_{i \in s} \lambda_{i} Z \mu_{i} \leq Z \mu$, and distinguish the two cases $\alpha < 1$ and $\alpha >1$. When $\alpha < 1$, the function $x^{\alpha - 1}$ is decreasing on $(0, \infty)$ but $1/(1-\alpha) >0$. On the other hand, when $\alpha > 1$, the function $x^{\alpha - 1}$ is increasing on $(0, \infty)$ but $1/(1-\alpha) <0$. 
Thus, setting $\lambda_s=\sum_{i\in s} \lambda_i$, we have 
\begin{align*}
D_{\alpha}(\mu)
& \leq \left( \sum_{s\in 2^{[n]}} \beta(s) \sum_{i \in s} \bigg\{ \lambda_i \int ( \lambda_s Z \mu_s)^{\alpha-1} \d \mu_i \bigg\} \right) ^{1/(1-\alpha)} \\
&= \left( \sum_{s\in 2^{[n]}} \beta(s)  \lambda_s \int ( \lambda_s Z \mu_s)^{\alpha-1} \d \mu_s  \right) ^{1/(1-\alpha)} .
\end{align*}

Note that for $\alpha \in [0,1) \cup (1, \infty)$, $x^{1/(1-\alpha)}$ is convex, and 
\[ \sum_{s\in 2^{[n]}} \beta(s) \lambda_s=1.\]
Therefore, by an application of Jensen's inequality,
\begin{align*}
 D_{\alpha}(\mu)
& \leq   \sum_{s\in 2^{[n]}} \beta(s) \lambda_s \left[ \int (\lambda_s Z \mu_s)^{\alpha-1} \d \mu_s \right] ^{1/(1-\alpha)} \\
& =   \sum_{s\in 2^{[n]}} \beta(s)\left[ \int (Z \mu_s)^{\alpha-1} \d \mu_s \right] ^{1/(1-\alpha)} \\
&=\sum_{s \in 2^{[n]}} \beta(s) D_{\alpha}(\mu_s).
\end{align*}
The cases $\alpha=1$ and $\alpha=\infty$ follow by taking limits.
\end{proof}

We now obtain as a consequence the fractional subadditivity result of Theorem~\ref{thm: subadditivecomplex} for metric complexity.

\begin{proof}[Proof of Theorem \ref{thm: subadditivecomplex}]
Note that the proof of the theorem is straightforward from the previous proposition when the $A_{i}$ are disjoint, since any $\mu \in \mathcal{P}(\cup_{i=1}^{n} A_{i})$ in this case decomposes naturally into a mixture of $\mu_{i} \in \mathcal{P}(A_{i})$ by the law of total probability.

Now suppose $F \subseteq \cup_{i \in [n]} F_{i}$ and all sets are finite. Let $\widetilde{F}_{i} = F_{i} \setminus \cup_{k = 1}^{i-1} F_{k}$ denote the ``disjointification'' of the cover. Then $F \subseteq \cup_{i \in [n]} \widetilde{F}_{i}$, as well as, $\cup_{i \in s} \widetilde{F}_{i} \subseteq \cup_{i \in s} F_{i}$ for every $s \in 2^{[n]}$. By the theorem for disjoint sets and monotonicity of complexity
with respect to inclusion, 
\[
D(F) \leq \sum_{s \in 2^{[n]}} \beta(s) D \left( \bigcup_{i \in s } \widetilde{F}_{i} \right) \leq \sum_{s \in 2^{[n]}} \beta(s) D \left( \bigcup_{i \in s } F_{i} \right).
\]
Finally, consider compact sets $A, A_{1} , \ldots , A_{n}$ such that $A \subseteq \bigcup_{i} A_{i}$. For any fixed finite $F \subseteq A$, set $F_{i} = A_{i} \cap F$. Then $F \subseteq F_{i}$ and so, 
\[
D(F) \leq \sum_{s \in 2^{[n]}} \beta (s) D \left( \bigcup_{i \in s} F_{i} \right) \leq \sum_{s \in 2^{[n]}} \beta (s) D \left( \bigcup_{i \in s} A_{i}\right).
\]
Taking supremum over all finite $F \subseteq A$ finishes the proof.
\end{proof}

\section{Discussion}
\label{sec:concl}

We conclude with some remarks and questions.

\begin{enumerate}
    \item The subadditivity property of $\kappa^{t}$ for wedge sums of pointed metric spaces, which is stated in Theorem~\ref{thm: wedgediv} and arises in our proof of Theorem \ref{thm: main}, may be of independent interest.  It would be interesting to know if this property 
    can be extended to submodularity.
\begin{Question}
   Let $A, B, C$, be non-empty finite pointed metric spaces. Then, under what additional conditions do we have
   \begin{equation} \label{eq: submodularityQ}
        \kappa^{t} (A \vee B \vee C) + \kappa^{t} (C) \leq \kappa^{t} (A \vee C) + \kappa^{t} (B \vee C)?
   \end{equation}
\end{Question}
 In the language of metric complexity, Equation \eqref{eq: submodularityQ} asks for the same inequality:  $\exp C^{t} (A \vee B \vee C) + \exp C^{t} (C) \leq \exp C^{t} (A \vee C) + \exp C^{t} (B \vee C)$. Denote \emph{magnitude} by $M^{t}$ (we will not define it here, but refer the reader to \cite{LM17:1} where $M^{t}(A)$ is denoted $\vert t \cdot A \vert$). Then, under the assumption that $M^{t}(A), M^{t}(B), M^{t}(C)$ are well-defined (which, unlike for metric complexity, may not always be the case), the equality $M^{t} (A \vee B \vee C) + M^{t} (C) = M^{t} (A \vee C) + M^{t} (B \vee C)$ holds. This is a simple consequence of the corresponding property when $C$ is a singleton \cite[Corollary 2.3.3]{Lei13} and associativity of the wedge sum.

 We also remark that submodularity properties of related quantities such as entropies, cardinalities and volumes under various operations have been extensively studied (see, e.g., \cite{Fuj78,Mad08:itw,FMZ24,FMMZ24}).

\item Proposition~\ref{prop: fracsubmultcomplex} may be stated as follows: for any $\alpha\in [0,\infty]$, the exponentiated $\alpha$-complexity of a mixture is fractionally subadditive. An immediate consequence, by considering the scaled metric $td$ and letting $t\rightarrow\infty$ (or equivalently by taking the similarity matrix $Z$ to be the identity matrix), is that the exponentiated $\alpha$-R\'enyi entropies of mixtures (of probability distributions on a finite alphabet, say) are fractionally subadditive. As far as we know, this fact has not been observed in the literature, and may well find fruitful applications in information theory. 

Indeed, the simplest case of Proposition 3.1, namely, for R\'enyi entropy with $n=2$, $\lambda_1=\lambda_2=1/2$, and $\beta(\{1\})=\beta(\{2\})=1$, implies that
\begin{align}\label{rev-conc-renyi}
H_\alpha\bigg(\frac{\mu_1+\mu_2}{2}\bigg)
&\leq \log \bigg[ e^{H_\alpha(\mu_1)}+e^{H_\alpha(\mu_2)}\bigg] .
\end{align}
For $\alpha=1$, it is a classical fact (see, e.g., \cite{MTBMS22}) complementing the concavity of the Shannon entropy that 
$$
H_1\bigg(\frac{\mu_1+\mu_2}{2}\bigg)
\leq \log 2 +  \bigg[ \frac{H_1(\mu_1)+H_1(\mu_2)}{2}\bigg] .
$$
The inequality \eqref{rev-conc-renyi} provides an extension of this inequality for mixtures from the Shannon entropy to R\'enyi entropies of any order $\alpha\in [0,\infty]$,
but at the cost of a weaker upper bound; note by the concavity of the logarithm and Jensen's inequality that 
$$
\frac{H_\alpha(\mu_1)+H_\alpha(\mu_2)}{2}
\leq \log\bigg[\frac{e^{H_\alpha(\mu_1)}+e^{H_\alpha(\mu_2)}}{2}\bigg].
$$

\item Theorem~\ref{thm: subadditivecomplex} says that the exponentiated complexity of compact sets is 
fractionally subadditive with respect to unions. It is a curious fact proved in \cite{BM24} that for compact subsets of a Euclidean space, the volume functional (which is akin to exponentiated complexity) is fractionally {\it superadditive} with respect to Minkowski summation.

\item   The Cauchy--Davenport inequality (for $\R^n$) states that if $A , B \subseteq \R^n$ are non-empty finite sets, then $\# (A+ B) \geq \# (A) + \# (B) - 1$, which implies $\kappa^{\infty} (A + B) \geq \kappa^{\infty}(A) + \kappa^{\infty} (B)$ for the diversity $\kappa^{\infty} (\cdot)$ defined by $\max \{ \# (\cdot ) - 1 , 0 \}$. Thus, when $n=1$, the inequality \eqref{eq: CauchyDavenport} is a generalization of the Cauchy--Davenport inequality. It is a natural question whether our result extends to dimensions $n > 1$. A similar comment can be made regarding the generalization of the one-dimensional Brunn--Minkowski inequality contained in \eqref{eq: BrunnMinkowski}. We note that R\'enyi entropy versions of the Cauchy--Davenport inequality have been considered on the integers in \cite{MWW19, MWW21}.

\item  Proposition \ref{prop: fracsubmultcomplex} is true under bare assumptions on the kernel $Z$, as evident in its proof. For simplicity, we state it only in the metric setting with the Laplace kernel. 

\item Theorem \ref{thm: subadditivecomplex} has a nice interpretation for finite metric spaces. Let us first note that, by considering the fractional partition 
$$\beta(s)=
\bigg\{\begin{array}{cc} 
\frac{1}{n-1}, &\quad |s|=n-1 \\
0, &\quad \text{otherwise}
\end{array},
$$
Theorem \ref{thm: subadditivecomplex} implies that if $A$ is a compact metric space and $\{A_1, \ldots, A_n\}$ is a collection of subsets that cover $A$ (i.e., $A=\cup_{i\in [n]} A_i$), then
$$
\frac{D(A)}{n} \leq \frac{1}{n} \sum_{i\in [n]} \frac{D(\cup_{j\neq i} A_j)}{n-1} .
$$
Consider the special case where $A$ is a finite metric space of cardinality $n$, and the collection $\{A_1, \ldots, A_n\}$ is the collection of all singletons. Then the inequality above says that the ``complexity per element'' of $A$ is at most the average ``complexity per element'' of a randomly drawn subset of size $n-1$. In other words, it captures the very natural intuition that the complexity per element decreases on average as the cardinality of the metric space increases.
\end{enumerate}

\bibliographystyle{amsplain}
\bibliography{pustak}
\end{document}